\documentclass[11pt]{article}
\usepackage[latin1]{inputenc}
\usepackage[english]{babel}
\usepackage{amssymb}
\usepackage{amsmath,amsthm}
\usepackage[normalem]{ulem}
\usepackage{cite}

\usepackage{physics}

\newtheorem{thm}{Theorem}[section]

\newtheorem{cor}{Corollary}[section]
\newtheorem{lem}{Lemma}[section]
\newtheorem{rmk}{Remark}[section]
\newtheorem{exmp}{Example}[section]

\makeatletter

\let \be=\beta
\let \var=\phi
\let \vare=\varepsilon

\let \de=\delta
\let \th=\theta

\let \ga=\gamma
\let \p=\partial
\let \q=\quad
\let \qq=\qquad
\let \med=\medskip
\let \smal=\smallskip
\let \dps=\displaystyle

\newcommand{\R}{\mathbb{R}}
\newcommand{\N}{\mathbb{N}}

\begin{document}


\vspace{0.2cm}

\begin{center}
\textbf{\Large{Global attractivity  for a nonautonomous Nicholson's equation with mixed monotonicities}}
	\end{center}

\centerline{\scshape Teresa Faria\footnote{Corresponding author.
E-mail:~teresa.faria@fc.ul.pt.}}
\smallskip
{\footnotesize
 \centerline{Departamento de Matem\'atica and CMAFcIO,}
   \centerline{ Faculdade de Ci\^encias, Universidade de Lisboa, Campo Grande, 1749-016 Lisboa, Portugal}}

\

\centerline{\scshape Henrique C. Prates\footnote{Research funded by a BII fellowship for the project UIDB/04561/2020 (CMAFcIO).}}
\smallskip
\smallskip
{\footnotesize
   \centerline{ Faculdade de Ci\^encias, Universidade de Lisboa, Campo Grande, 1749-016 Lisboa, Portugal}}

 
%
%
%
\vskip .5cm

\begin{abstract}  
We consider a Nicholson's  equation with multiple pairs of time-varying delays and nonlinear terms given by mixed monotone functions. Sufficient conditions for the permanence, local stability and global attractivity of its positive equilibrium are established. Our criteria depend on the size of some  delays,  improve  results in recent literature and provide  answers to some open problems.

\end{abstract}



\section{Introduction}
\setcounter{equation}{0}

The classical Nicholson's blowflies equation \begin{equation}\label{N}
    x'(t)=- \delta x(t)+ p x(t-\tau)e^{-a x(t-\tau)} \q (p,\de,\tau>0)
  \end{equation}
  has been extensively studied since its introduction by Gurney et al.~\cite{GBN}, and many scalar and multi-dimensional variants analysed in different contexts of mathematical biology. 
  See e.g. \cite{bb2017,smith, NE_mult, MR18} for biological explanations of the model and some additional references.
  
Recently, a Nicholson's model  with two different delays
\begin{equation} \label{Nich_MR}
  x'(t)=- \delta x(t)+ p x(t-\tau)e^{-a x(t-\sigma)},
  \end{equation}
  where  $p>\de>0$ and  $\tau\ge\sigma\ge 0$, was proposed and studied by  El-Morshedy and  Ruiz-Herrera \cite{MR18}, and  a criterion for  the global attractivity of its positive equilibrium $K=\frac{1}{a}\log (p/\de)$  established. The Nicholson's model \eqref{Nich_MR}  involves  a so-called {\it  mixed monotonicity} nonlinear term. To be more concrete, in \eqref{Nich_MR}
the nonlinear term is given by  $g(x(t-\tau),x(t-\sigma))$, where $g(x,y)=pxe^{-ay}$ is  monotone increasing in the first variable and monotone decreasing in the second. 


In recent years, there has been a growing interest in DDEs with  two or more delays appearing  in the same nonlinear function, possibly with a mixed monotonicity. For delay differential equations (DDEs), the presence of a unique small delay  in each nonlinear term is in general harmless, i.e.,  the  global  dynamics properties of the equation without delays are kept under small delays. However, 	  the  situation may be drastically different
if  two or more different delays are involved in each  nonlinearity, 
as exemplified  in \cite{bb2017} for a variant of the celebrated  Mackey-Glass  equation. How to generalise known results for classical DDEs to some modified versions  with two or more different delays 
 in the same nonlinear term is a  subject that lately has caught the attention of a number of researchers, since such models appear naturally in real-world problems. Several different techniques have been proposed, such as  Lyapunov functions or the theory of monotone systems, which are difficult to apply in the context of mixed monotonicity. The method in \cite{MR18}, based on an auxiliary monotone  difference equation associated with \eqref{Nich_MR}, has inspired our work. 
 However, the application of this method to  (possibly non-autonomous) equations with more than one pair of  delays is not apparent, and has not been addressed.
 

The following non-autonomous version of \eqref{Nich_MR} with time-varying delays  has also been suggested:
\begin{equation} \label{eq:nicholson_1pair}
    x'(t)=\beta(t) (p x(t-\tau(t))e^{-a x(t-\sigma(t))}-\delta x(t)) ,
\end{equation}
where $p, a, \delta \in (0, \infty  )$, $\be(t)> 0$, and $ \tau(t),\sigma(t)$ are non-negative and bounded. To obtain conditions for  the stability of the equilibria for this equation has been an open problem   \cite{bb2017}.  In fact, Berezansky and Braveman   \cite{bb2017} provided criteria for the local and global stability of the positive equilibrium    of the  Mackey-Glass-type equation   $x'(t)=\be(t)\Big [- x(t)+ \frac{a x(t-\tau(t))}{1+x^\nu(t-\sigma(t))}\Big]$
(with   $\nu>0, a>1$), and set  as an open problem a similar analysis for the Nicholson model \eqref{eq:nicholson_1pair}.

More recently, Long and Gong \cite{NE_mult} considered  a Nicholson's equation with multiple pairs of time-varying delays of the form
\begin{equation}\label{eq:nicholson}
    x'(t)=\beta (t) \bigg( \sum_{j=1}^{m} p_j  x(t-\tau_j(t)) e^ {-a_j x(t-\sigma_j(t))}-\delta x(t)\bigg)\, , \qq t\ge t_0
\end{equation}
where $p_j , a, \delta \in (0, \infty )$ and $\beta(t), \sigma_j(t), \tau_j(t)$ are continuous, non-negative and bounded, with $\be(t)$ bounded from below by a positive constant. 
If $\sum_j p_j \le\delta$, it was shown in \cite{NE_mult} that the equilibrium $0$  is a global attractor of all positive solutions, and moreover is globally exponentially stable if $\sum_j p_j  <\de$. Typically, the growth model \eqref{eq:nicholson} is used in population dynamics,   regulatory physiological systems and  other contexts, where this situation corresponds to the extinction of the population represented by $x(t)$.

The study of the global attractivity of a positive equilibrium, when it exists, is more interesting from a biological viewpoint. 
Since $f(x):=\de^{-1} \sum_j p_j  e^{-a_j x}$ is strictly decreasing with
$f(0)=\de^{-1}\sum_j p_j $ and $ \lim_{x \to +\infty} f(x) = 0,$
there is a unique positive equilibrium $K$ if and only if $ \sum_{j=1}^m  p_j  > \delta$.
In the present paper, we shall always assume $ \sum_{j=1}^m  p_j  > \delta$,
in which case the {\it carrying capacity} $K>0$ is defined by the identity
\begin{equation}\label{equilK} 
 p:=   \sum_{j=1}^m p_j  e^{-a_j K}=\delta.
\end{equation}
 For  \eqref{eq:nicholson_1pair}, if $p>\delta$ the positive
equilibrium is explicitly given by
 \begin{equation}\label{Kexplicit}K=\frac{1}{a} \log (\frac{p}{\delta}).\end{equation}
 
 In a recent work, Huang et al. \cite{huang20} have also addressed the global attractivity of the positive equilibrium of the ``neoclassical growth model" obtained by replacing in \eqref{eq:nicholson} $ x(t-\tau_j(t)) $ by $ x^\ga(t-\tau_j(t)) \ (1\le j\le m)$ for some fixed $\ga\in (0,1)$. See \cite{huang20} for additional references on real world applications of such models.
 
 Motivated by the above mentioned works \cite{NE_mult, MR18}, the main goal of this paper is to establish sufficient conditions for the equilibrium $K$ of \eqref{eq:nicholson} to be a global attractor of all positive solutions.

To this end, we introduce briefly some notation and  the abstract framework to deal with the DDE model \eqref{eq:nicholson}. In what follows, we always assume the  general assumptions below:
\begin{itemize}
\item[] $p_j , a_j, \delta \in (0, \infty )$,  $\beta, \sigma_j, \tau_j:[t_0, \infty )\to [0, \infty )$ are continuous and bounded ($1\le j\le m$), 
with
\begin{equation}\label{eq:beta_def}
0<    \beta^- :=\inf_{t\ge t_0} \beta(t) \le \beta(t) \le \sup_{t\ge t_0} \beta(t)=: \beta^+ .
    \end{equation}
 \end{itemize} 
 In what follows, denote
$$p=\sum_{j=1}^{m} p_j ,\q a^+=\max_j a_j,\q a^-=\min_j a_j,$$
assume $p>\de$
 and let $K>0$ be the equilibrium defined by  equation \eqref{equilK}. 
Define
$$\tau=\max\{\sup_{t\ge t_0} \tau_j(t), \sup_{t\ge t_0} \sigma_j(t):j=1,\dots,m\}.$$
For \eqref{eq:nicholson}, the  space $C:=C([-\tau, 0];\R)$ equipped with  the supremum norm  $\|\phi\|=\max_{\th\in[-\tau,0]}|\phi(\th)|$ will be taken as the phase space \cite{smith}.
Furthermore, due to the biological motivation of the model, we are only interested in non-negative solutions of \eqref{eq:nicholson}. Thus, we  consider 
$C_0^+:=\big\{\phi\in C : \phi(\th)\ge 0\ {\rm on}\ [-\tau,0),  \phi(0)>0 \}$ as the set of admissible initial conditions. For $t\ge t_0$,  $x_t$ as usual designates  the element in $C$  given by $x_t(\th)=x(t+\th), -\tau\le \th\le 0$. In this way,  we write $x_t(t_0, \phi)\in C$, or $x(t; t_0, \phi)\in \R$, to denote the solution of \eqref{eq:nicholson} with the initial condition 
\begin{equation}\label{eq:IC}
    x_{t_0}=\phi\in C_0^+,
\end{equation}
defined on a maximal interval $[t_0, \eta(\phi))$.
Presently, the usual concepts of stability  always refer to solutions with initial conditions in $C_0^+$. In this way, the equilibrium $K$ is called a {\it global attractor} of  \eqref{eq:nicholson} if it is globally attractive in $C_0^+$, i.e.,
$\lim_{t\to\infty}x(t)=K$
for every solution $x(t)=x(t;t_0,\phi)$ of \eqref{eq:nicholson}  with initial condition \eqref{eq:IC}; $K$ is {\it gobally asymptotically stable} if it is stable and globally attractive.

The paper is organized as follows. With $p >\de$,  in Section 2  we show that  \eqref{eq:nicholson} is permanent  and
 derive sufficient conditions for  the local asymptotic stability of its positive equilibrium $K$. Relying on the permanence previously established and with  the methodology in  \cite{MR18} as a key ingredient, the main resuls  on the global attractivity of $K$ are obtained in Section 3. Our criteria   largely generalise the results obtained in \cite{MR18} for \eqref{Nich_MR}
and in particular provide an  answer to the above mentioned open problem concerning \eqref{eq:nicholson_1pair} raised in \cite{bb2017}.  A couple of examples illustrate our results.
Some conclusions end the paper.

\section{Permanence and local stability}\label{sec:permanence}
\setcounter{equation}{0}

 In abstract form, \eqref{eq:nicholson} is written as
$x'(t)=F(t,x_t),\q t\ge t_0,$ where
$$F(t,\phi)=\beta (t) \bigg( \sum_{j=1}^{m} p_j  \phi(-\tau_j(t)) e^ {-a_j \phi(-\sigma_j(t))}-\delta \phi(0)\bigg),\q t\ge t_0,\phi\in C_0^+.$$
For $t\ge t_0,\phi\in C_0^+$, $F(t,\phi)\ge 0$ if $\phi(0)=0$ and $F(t,\phi)\le g(t,\phi)$, where the function
$ g(t,\phi):=\beta (t) \Big( \sum_{j=1}^{m} p_j  \phi(-\tau_j(t))- \delta \phi(0)\Big)$ satisfies the quasi-monotone condition in \cite[p.~78]{smith1}. From results of comparison of solutions \cite{smith1}, solutions of \eqref{eq:nicholson} with initial conditions \eqref{eq:IC} are  defined and positive on $[t_0,\infty)$.

We recall the standard definitions of (uniform) persistence and permanence. The DDE  \eqref{eq:nicholson} is said to be {\it persistant} if 
 there exists  $m>0$ such that for every $\var\in C_0^+$ there is $ t_*=t_*(\phi)$ such that
 the solution $x(t; t_0, \phi)$ satisfies \begin{equation*}
 x(t;t_0, \phi)\ge m\q {\rm for} \q  t \ge t_*.
\end{equation*}
Eq. \eqref{eq:nicholson} is said to be {\it permanent} if it is both dissipative and persistant, i.e., there are constants $m,M>0$ such that for $\var\in C_0^+$ there is $ t_*=t_*(\phi)$ such that
 \begin{equation*}
m\le  x(t;t_0, \phi)\le M\q {\rm for} \q  t \ge t_*.
\end{equation*}

We now establish the permanence of  \eqref{eq:nicholson}, noting however that we are mostly interested in the permanence as an auxiliary result to derive conditions for the global attractivity of $K$. 
\begin{thm} \label{thm:permanence} If $p>\delta$,
then \eqref{eq:nicholson} is permanent. Moreover, 
\begin{equation}\label{perm}
     Ke^{-2\delta \beta^+ \tau} \le \liminf_{t\to \infty}x(t)\le \limsup_{t\to \infty}x(t) \le Ke^{2(\delta + p) \beta^+ \tau}
\end{equation}
 for any solution $x(t)$ of \eqref{eq:nicholson} with initial condition in $C_0^+$.
\end{thm}
\begin{proof}
We rewrite \eqref{eq:nicholson} as 
\begin{equation*}
    x'(t)=\sum_j f_j\big(t, x(t-\tau_j(t)), x(t-\sigma_j(t))\big) - w(t, x(t))
\end{equation*}
where $f_j(t, u, v)=\beta(t)p_j  u e^{-a_j v}$ and $w(t, u)=\delta \beta(t) u$. 
The functions $f_j$ are monotone increasing in $u$ and monotone decreasing in $v$. Furthermore, 
\begin{equation*}
    \sum_j f_j(t, u_j, v_j) = \beta(t) \sum_j p_j  u_j e^{-a_j v_j} \le \beta^+ \sum_j p_j  \Bar{u} \le A \Bar{u}\end{equation*}
where $ \Bar{u}=\max_{1\le j\le m} \{u_j \} $, $A:= \beta^+ p >0,$
and
\begin{equation*}
    B u \le w(t, u) \le C u \q{\rm where }\q B:= \delta \beta^- ,\ C:= \delta \beta^+.
\end{equation*}
Let $M>0$. Then,
\begin{equation*}
     L=L(M):=\limsup_{t \to \infty } \frac{\sum_j f_j(t, u, M)}{w(t, u)}
     = \frac{\sum_j p_j  e^{-a_j M}}{\delta}.
\end{equation*}
Thus, $L<1$ as long as $M>K$. From \cite[Theorem 5.6]{bb2016}, for any $M>K$ we obtain
\begin{equation*}
    \limsup_{t\to \infty}x(t) \le M e^{2(\delta+p)\beta^+ \tau},
\end{equation*}
yielding that \eqref{eq:nicholson} is dissipative.
Next, for $\mu>0$ we have
\begin{equation*}
    l=l(\mu):=\liminf_{t \to \infty } \frac{\sum_j f_j(t, u, \mu)}{w(t, u)}
    =\frac{\sum_j p_j  e^{-a_j \mu}}{\delta}.
\end{equation*}
Thus, $l>1$ if $\mu<K$. From \cite[Theorem 5.6]{bb2016}, for $\mu <K$ we derive
\begin{equation*}
    \liminf_{t\to \infty}x(t)\ge \mu e^{-2\delta \beta^+ \tau},
\end{equation*}
hence \eqref{eq:nicholson} is persistant.
Combining these results for $M=K+\varepsilon$ and $\mu =K-\varepsilon$ and letting $\vare\to 0^+$, \eqref{perm} follows.
\end{proof}

 In the case of a single pair of time-varying delays as in \eqref{eq:nicholson_1pair}, instead of  \cite{bb2016}, see  \cite{GHM} for an alternative approach, with possible better uniform lower and upper bounds then the ones  in \eqref{perm}.

\begin{rmk}{\rm Theorem \ref{thm:permanence} largely extends the permanence result obtained in \cite{MR18} for the Nicholson's equation \eqref{Nich_MR}, where, besides  a single pair of {\it constant} delays $\tau(t)\equiv \tau,\sigma(t)\equiv \sigma$, the additional constraint $0\le \sigma\le \tau$ is imposed.}\end{rmk}

%

 The local stability of the positive equilibrium $K$ is now studied. 
By setting $u(t)= x(t)-K$ and dropping the nonlinear terms, one obtains 
 the linearised equation about $K$ given by
\begin{equation}\label{eq:nicholson_lin}
   u'(t)= \beta (t) \bigg( \sum_j p_j  e^{-a_j K} [u(t-\tau_j(t)) -a_j K u(t-\sigma_j(t))]-\delta u(t) \bigg).
\end{equation}
\begin{thm}\label{thmLAS2}
If $p >\delta$ and 
\begin{equation}\label{eq:las_K}
    \limsup_{t \to +\infty}  \sum_{j=1}^m
    a_j p_j  e^{-a_j K}\int_{t-\sigma_j(t)}^{t} \beta(s)\, \dd s <\frac{\sum_{j=1}^m a_j p_j  e^{-a_j K}}{2\delta + K\sum_{j=1}^m a_j p_j  e^{-a_j K}},
\end{equation}
then $K$ is locally exponentially stable for \eqref{eq:nicholson}.
\end{thm}
\begin{proof}
Write \eqref{eq:nicholson_lin} as 
\begin{equation*}
    u'(t)+\sum_{j=0}^{2m} b_j(t) u(r_j(t))=0
\end{equation*}
where $b_j(t)=B_j \beta(t)$,
\begin{align*}
    & B_j=\begin{cases}
    \delta & j=0\\
    - p_j  e^{-a_j K} & j=1, ..., m\\
      a_jK  p_j  e^{-a_j K} & j=m+1, ..., 2m
    \end{cases},
    &r_j(t)=\begin{cases}
    t &j=0\\
    t-\tau_j(t) &j=1, ..., m\\
    t-\sigma_j(t) &j=m+1, ..., 2m
    \end{cases},
\end{align*}
 for  all $j=0, ..., 2m$ and  $\beta(t)\ge \beta^- >0$. Note that \eqref{eq:nicholson_lin} has the form required in \cite[Corollary 1.3]{bb2006}. Let $I=\{0, m+1, ..., 2m\}$ and $J=\{1, ..., m\}$. Then,
\begin{equation*}
    \sum_{j\in I} B_j=\delta +K\sum_{j=1}^m a_j p_j  e^{-a_j K}>0, \q \sum_{j\in J} B_j=-\delta.
\end{equation*}
Following \cite{bb2006}, we have exponential stability for \eqref{eq:nicholson_lin} if
\begin{equation*}
    \limsup_{t \to +\infty} \sum_{j\in I} \abs{B_j}\int_{r_j(t)}^t \beta(s)\, \dd s < \frac{\sum_{j\in I}B_j-\sum_{j\in J}\abs{B_j}}{\sum_{j=0}^{2m} \abs{B_j}}.
\end{equation*}
Hence, \eqref{eq:nicholson_lin} is exponentially stable if  \eqref{eq:las_K} holds.
By the principle of linearized stability \cite{smith}, this yields  the  result.
\end{proof}

Considering an uniform upper bound in the integrals in \eqref{eq:las_K}, we obtain the criterion below.
\begin{cor}\label{cor3.1}
Let $p >\delta$ and define
\begin{equation}\label{zeta_M}
\zeta_M =\max_{1\le j\le m} \limsup_{t \to \infty} \int_{t-\sigma_j(t)}^{t}  \beta(s)\, \dd s.
\end{equation}
If
\begin{equation*}
  \zeta_M<\frac{1}{2\delta+K(\sum_j p_j a_je^{-a_jK})},
\end{equation*}
then $K$ is locally exponentially stable for \eqref{eq:nicholson}. In particular, this holds if
\begin{equation}\label{eq:las_K2}
 \zeta_M<\frac{1}{\delta(2+a^+K)}.\end{equation}
\end{cor}

In the case of only one pair of mixed delays, from the identity $aK=\log(p/\de)$ the next corollary follows immediately. For an alternative proof, one can use \cite[Lemma 3.4]{bb2017}.

\begin{cor}\label{corLAS}
If $p>\delta$ and 
\begin{equation}\label{cond:las}
\zeta_M:=\limsup_{t \to \infty }\int_{t-\sigma(t)}^{t} \beta (s)\, \dd s < \frac{1}{ \delta(2+\log(p/\de))}\, ,
\end{equation} then $K$ is locally exponentially stable for \eqref{eq:nicholson_1pair}.
\end{cor}


\section
{Global  attractivity of the positive equilibrium }
\setcounter{equation}{0}

The next goal is to establish sufficient conditions for the global attractivity of the positive equilibrium $K$ (in the set of all positive solutions). 

As in \cite{MR18}, to preclude the existence of  solutions for which $\liminf_{t\to\infty}x(t)<\limsup_{t\to\infty}x(t)$, 
  we shall use global attractivity results for difference equations of the form
\begin{equation}\label{diffeq}
x_{n+1}=f(x_n),\qq n\in \N_0,
\end{equation} 
where $f:I\to I$ is a continuous function on a real interval $I$ and the initial data $x_0\in I$. 

For a compact interval $I=[a,b]$ ($a<b$), the pioneering work of Coppel \cite{coppel}   shows that if the equation $f^2(x)=x$ (where $f^2=f\circ f$) has no roots except the roots of $f(x)=x$, then all solutions of the difference equation \eqref{diffeq}
converge to a fixed point of $f$. In particular, if there is a unique fixed point $x^*$ of $f$, then $x^*$ is a global attractor for
\eqref{diffeq}.
In applications, however, it is often rather difficult to find the fixed points of $f^2$. In the case of $C^1$ functions with a unique fixed point $x^*$, the assumption $|f'(x)|<1$  would show that $x^*$ is the unique fixed point of $f^2$, but again this condition is usually not checkable in a simple way. Lately,  more powerful  criteria  (and easier to verify) were given  in \cite{ML}. For convenience of the reader, we include below some selected  results derived from  \cite[Lemma 2.5 and Corollary 2.9]{ML}. (With $a\in\R$ and $b=\infty$, $I=[a,b]$ denotes the interval $[a,\infty)$.)

\begin{lem}\label{lemML}\cite{ML}
Let $f:I\to I$ be a continuous function on a real interval $I=[a,b]$ ($a<b$ with $b\in \R$ or $b=\infty$) with 
a unique fixed point $x^*\in (a,b)$ and such that
$$(f(x)-x^*)(x-x^*)<0\qq {\rm for \ all}\ x\in (a,b), x\ne x^*.$$ 
The following assertions hold:
\begin{enumerate}
\item[(a)] If $x^*$ is a global attractor for \eqref{diffeq}, then  there are no points $c,d\in I$ with $c< d$ and $f([c,d])\supset [c,d]$.
\item[(b)] If $f$ is a $C^3$  function
and  there are constants $c,d\in I$ with $c<x^*<d$ such that:\vskip 0mm
(i)  $Sf(x)<0$ for $x\in (c,d)$, where  $Sf(x)$ is the Schwarzian derivative of $f$, defined by
$$ Sf(x):=\frac{f'''(x)}{f'(x)}-\frac{3}{2}\left(\frac{f''(x)}{f'(x)}\right)^2,$$
 (ii) $-1\le f'(x^*)<0$,\vskip 1mm
 then $x^*$ is a global attractor for \eqref{diffeq}.
\end{enumerate}
\end{lem}

We are now in the position to state the main results of this section on the global attractivity of $K$. The proofs are inspired in the techniques of El-Morshedy and Ruiz-Herrera \cite{MR18}, however  we need to carefully adjust the major arguments in \cite{MR18} to the present situation. 

Recall that it is not possible to write explicitly the  equilibrium $K$ if $m>1$, see \eqref{equilK}.  This alone justifies considering separately, for the sake of simplicity of exposition, the case of a single pair of delays, prior to the general situation \eqref{eq:nicholson}.




\subsection{The case of one pair of delays}

\begin{thm}\label{thmGA_K1}
Assume $p>\de$,   denote $\dps \zeta_M := \limsup_{t \to \infty} \int_{t-\sigma(t)}^{t} \beta(s)\, \dd s$ as in \eqref{cond:las} (with $m=1$),
and further assume 
\begin{itemize}
\item[(h1)] $\dps (e^{ \delta\zeta_M} -1)\log \frac{p}{\delta}  \le 1.$
 \end{itemize}  
%
%
Then the equilibrium $K$   of \eqref{eq:nicholson_1pair} is globally attractive (in $C_0^+$).
\end{thm}
\begin{proof} Consider any solution of \eqref{eq:nicholson_1pair} with initial condition in $C_0^+$.
We need to prove that there exists $ \lim_{t\to \infty} x(t)=K $.
By the permanence in Theorem \ref{thm:permanence},  there exists
\begin{equation}\label{lL}
    l:= \liminf_{t \to \infty }x(t) \q {\rm and}\q   \limsup_{t \to \infty }x(t)=: L
\end{equation}
with $0<l\le L$.
\smal

{\it Case 1}: 
 Suppose that $l=L$, i.e., there exists $ \lim_{t\to \infty} x(t)=C>0$.
 \smal

From \eqref{eq:nicholson_1pair}, we may write $x'(t)=\be(t)g(x_t)$, where
$\lim_{t\to\infty}g(x_t)=\de C\big (f(C)-1\big )$
and $$f(x):=\frac{p}{\delta}e^{-ax},\q x\ge 0.$$
Let $\eta=f(C)-1$. If $0<C<K$, then for $t$ large we have $x'(t)\ge \be^-\de C\frac{\eta}{2}>0$, thus $x(t)\to \infty$ as $t\to \infty$, which is a contradiction. Similarly, $\eta<0$ if $C>K$, in which case  $x(t)\to -\infty$ as $t\to \infty$, which is not possible. Hence $C=K$.

%

\med

{\it Case 2}: Suppose that there is a solution   $x(t)$  with $l<L$ in \eqref{lL}.

\med

We preclude this situation  in several steps.
To simplify the exposition, we first suppose that there exists
\begin{equation}\label{zeta} \zeta := \lim_{t\to \infty } \int_{t-\sigma(t)}^t  \beta (s)\, \dd s,
 \end{equation}
leaving  the general situation, where $\zeta$ is replaced by $\zeta_M$,  for the last step of the proof.
 In this case, assumption (h1) reads as
 \begin{equation}\label{eq:condition3}
 (e^{\delta\zeta} -1)\log \frac{p}{\delta} \le 1,
\end{equation}

{\it Step 1.} As in the proof of \cite[Theorem 3.1]{MR18},  consider sequences $(t_n)$, $(s_n)$ such that $t_n, s_n \to \infty $ and $x'(t_n)= x'(s_n)=0$, $ x(t_n)\to L, x(s_n)\to l.$
Taking  subsequences if needed, there exist $l_\tau$, $l_\sigma, L_\tau$, $L_\sigma$ $\in[l, L]$ such that
\begin{equation}\label{limits_h,g}
    x(t_n-\tau(t_n))\to L_\tau, \q  x(t_n-\sigma(t_n))\to L_\sigma 
    , \q  x(s_n-\tau(s_n))\to l_\tau, \q x(s_n-\sigma(s_n))\to l_\sigma
\end{equation}
and $\lim_n \beta(t_n)>0,\lim_n \beta(s_n)>0$. From \eqref{eq:nicholson_1pair}, it follows that
\begin{equation*}
\begin{split}
    0&=\beta(t_n) \left(p x(t_n-\tau(t_n))e^{-ax(t_n-\sigma(t_n))}-\delta x(t_n)\right),\\
    0&=\beta(s_n) \left(p x(s_n-\tau(s_n))e^{-ax(s_n-\sigma(s_n))}-\delta x(s_n)\right).
    \end{split}
\end{equation*}
Using the limits in \eqref{limits_h,g}   leads to
  $ 0=p L_\tau e^{-a L_\sigma}-\delta L,\  0=p l_\tau e^{-a l_\sigma}-\delta l$,
thus
\begin{equation}
    L=\frac{p}{\delta}L_\tau e^{-a L_\sigma} \q {\rm and}\q   l=\frac{p}{\delta}l_\tau e^{-al_\sigma}.
\end{equation}
As $L_\tau \le L$ and $l_\tau \ge l$, we have  $L\le \frac{p}{\delta}L e^{-a L_\sigma}$ and $l \ge \frac{p}{\delta}l e^{-al_\sigma}$. Furthermore,
\begin{equation*}
    (\frac{p}{\delta} xe^{-ax} -x)(x-K) <0, \q x>0,x\ne K,
\end{equation*}
which combined with the above inequalities  yields 
\begin{equation}\label{limits_g}
L_\sigma \le K \le l_\sigma.
\end{equation}
By multiplying the equation by  $\mu(t)=e^{ \int_{t_0}^t \de \be(v)\, \dd v}$ and integrating over $[t-\sigma(t), t]$,  one obtains 
\begin{equation*}
    x(t)= x(t-\sigma(t)) e^{-\int_{t-\sigma(t)}^t \delta \beta(v)\, \dd v}+ \int_{t-\sigma(t)}^t p \beta(s) x(s-\tau(s))e^{-ax(s-\sigma(s))}e^{\int_t^s \delta \beta(v)\, \dd v} \, \dd s.
\end{equation*}
Let $\varepsilon>0$. There is  $n_0 \in \mathbb{N}$ such that $x(t)\in[l-\varepsilon, L+\varepsilon]$ for $ t\ge t_n-2\tau$ and $x(t_n-\sigma(t_n))\le L_\sigma+\vare$, for all  $n\ge n_0$. From the previous equation, it follows that
\begin{align*}
    x(t_n)&= x(t_n-\sigma(t_n)) e^{-\int_{t_n-\sigma(t_n)}^{t_n} \delta \beta(v)\, \dd v}+ \int_{t_n-\sigma(t_n)}^{t_n} p \beta(s) x(s-\tau(s))e^{-ax(s-\sigma(s))}e^{\int_{t_n}^s \delta \beta(v)\, \dd v} \dd s \\
    &\le (L_\sigma+\varepsilon) e^{-\int_{t_n-\sigma(t_n)}^{t_n} \delta \beta(v)\, \dd v}+ \int_{t_n-\sigma(t_n)}^{t_n} p \beta(s) (L+\varepsilon)e^{-ax(s-\sigma(s))}e^{\int_{t_n}^s \delta \beta(v)\, \dd v} \dd s.
\end{align*}
Using the mean value theorem for integrals, there exists $\tilde {l}_n(\varepsilon) \in [l-\varepsilon, L+\varepsilon]$ such that
$$x(t_n)\le (L_\sigma+\varepsilon) e^{-\int_{t_n-\sigma(t_n)}^{t_n} \delta \beta(v)\, \dd v}+ \frac{p}{\delta}(L+\varepsilon)e^{-a  \tilde l_n(\varepsilon)}\left(1-e^{-\int_{t_n-\sigma(t_n)}^{t_n} \delta \beta(v)\, \dd v}\right).$$
By taking subsequences if needed, we may suppose that $ \tilde l_n(\varepsilon) \to  \tilde l(\varepsilon)$ as $n\to\infty$. Taking  limits and using \eqref{zeta},  we obtain
\begin{equation*}
  L\le e^{-\de \zeta}(L_\sigma+\varepsilon)+\frac{p}{\delta}(L+\varepsilon)e^{-a  \tilde l(\varepsilon)}(1-e^{-\de \zeta}) . 
\end{equation*}
Once again, we may suppose that $ \tilde l(\varepsilon)\to  \tilde l$ as $\varepsilon\to 0$ and by passing to the limit it follows that
\begin{equation}\label{eq:ineqL}
    L\le e^{-\de \zeta}L_\sigma+\frac{p}{\delta}L e^{-a  \tilde l}(1-e^{-\de \zeta}) \le e^{-\de \zeta}K+\frac{p}{\delta}L e^{-a  \tilde l}(1-e^{-\de \zeta})
\end{equation}
Analogously, by using the sequence $(s_n)$, one also shows that
\begin{equation}\label{eq:ineql}
    l\ge e^{-\de \zeta}l_\sigma+\frac{p}{\delta}  l e^{-a \tilde L}(1-e^{-\de \zeta})\ge e^{-\de \zeta}K+\frac{p}{\delta}  l e^{-a \tilde L}(1-e^{-\de \zeta})
\end{equation}
where $\tilde L, \tilde l \in [l, L]$.

\smal
{\it Step 2.} We claim that $\zeta >0$ and $ l<K< L.$

We first observe that, if $\zeta=0$,   \eqref{limits_g},\eqref{eq:ineqL} and \eqref{eq:ineql} imply 
$$L\le L_\sigma\le K\le l_\sigma\le l,$$
then $l=L$, which contradicts our assumption.
Next, we show  that $\tilde l< K<\tilde L.$

If $\tilde l>K$, then $\frac{p}{\de}e^{-a\tilde l}<1$ and from \eqref{eq:ineqL} $L<e^{-\de \zeta}K+L(1-e^{-\de \zeta})\le L$, which is not possible. Hence $l\le \tilde l\le K$. Analogously,  from \eqref{eq:ineql} we deduce that $K\le \tilde L\le L.$

If $\tilde l=K$, again from \eqref{eq:ineqL} we derive
$$L\le e^{-\de \zeta}K+L (1-e^{-\de \zeta})\le L,$$
hence   $L=\tilde L=K$. Inserting this into \eqref{eq:ineql} yields
$ l\ge e^{-\de \zeta}K+l(1-e^{-\de \zeta})\ge l$, thus $l=K$, which  contradicts the assumption that $l<L$.
This shows that $\tilde l<K$. Similarly, one shows that  $\tilde L>K$. 
Thus, $$ l\le \tilde l<K< \tilde L\le L.$$

{\it Step 3.} 
From Steps 1 and 2, note that   $l> e^{-\de \zeta}K =: \theta$. Now, let  $f(x):=\frac{p}{\delta}e^{-ax}$ and $\mu:= 1-e^{-\de \zeta}=1-\frac{\theta}{K}$. Next, we prove that (h1) implies
\begin{equation}\label{eq:condition1}
    f(x)< \frac{1}{\mu}\, , \qq x\ge \theta. 
\end{equation} Since $f$ is decreasing, \eqref{eq:condition1} holds if $f(\th)<\mu^{-1}$, which is equivalent  to 
\begin{equation}\label{teta1}
    \th > \frac{1}{a}\log \frac{p}{\delta}+\frac{1}{a}\log (1-e^{-\de \zeta})=: \theta_1.
\end{equation}
 Inserting the identity $\theta=\frac{1}{a}\log \frac{p}{\delta}e^{-\de \zeta}$ in \eqref{teta1}, we have $\th>\th_1$ if and only if
\begin{equation*}
(1-e^{-\de \zeta})    \log\frac{p}{\delta}+\log(1-e^{-\de \zeta})<0.
\end{equation*}
Since
$ x+\log(1-x)<0 $ for all $x\in(0, 1)$, from \eqref{eq:condition3} we conclude that
  \begin{equation*}
 (1-e^{-\de \zeta})    \log\frac{p}{\delta}+\log(1-e^{-\de \zeta})\le   e^{-\de \zeta}+\log(1-e^{-\de \zeta})<0,
\end{equation*}
thus \eqref{eq:condition1} holds. 
Therefore, 
the function 
\begin{equation}\label{fc:h}
    h(x):=\frac{\th}{1-\mu f(x)} \end{equation}
is well-defined on $I:=[ \th,\infty).$ Note that $h(K)=K$ and $(h(x)-K)(x-K)<0$ for $x\ge \th, x\ne K$.
From  \eqref{eq:ineqL} and \eqref{eq:ineql}, it follows that
\begin{subequations} \label{eq:L_l_ineq}
    \begin{equation}
        L\le h( \tilde l)\le h(l)
    \end{equation}
    \begin{equation}
        l \ge h(\tilde L)\ge h(L).
    \end{equation}
\end{subequations}
We have $f \in C^3$, $f'(x)<0$ and $Sf(x)=-\frac{a^2}{2}<0$ for $x\ge \th$. Write $h(x)=h_1(f(x))$ where $h_1(x)=\frac{\th}{1-\mu x}$, and note that $Sh_1(x)\equiv 0$.  From the formula
$Sh(x)=Sh_1(f(x)) (f'(x))^2+Sf(x)$ (see e.g.~\cite{MR18}), we have $Sh(x)<0$ for all $x\ge \th.$ On the other hand,
$h'(K)= -aK(e^{\delta\zeta}-1)$, and therefore $-1\le h'(K)<0$ is equivalent to
\begin{equation*}\label{eq:condition2}
  0<  \log\frac{p}{\delta}\le \frac{1}{e^{\delta\zeta}-1}.
\end{equation*}
which, with $\zeta>0$, is our hypothesis (h1).
Thus, from Lemma \ref{lemML}(b),  
we derive that  $K$ is a global attractor of the difference equation\
    \begin{equation}
    x_{n+1}=h(x_n).
    \end{equation}
with initial conditions $x_0\ge \theta$.
From  Lemma \ref{lemML}(a), this implies that there are no points $c, d \in [\theta, \infty), c< d,$ such that $h([c, d]\supset [c, d]$.  This however contradicts  \eqref{eq:L_l_ineq}, since clearly $h([l, L])=[h(L),h(l)]\supset [l, L]$ with $l<K<L$. 
This shows that there are no solutions  with $l<L$  when the above  limit $\zeta$ in \eqref{zeta}   exists.

\med

{\it Step 4}.
The assumption that the limit in \eqref{zeta} exists is now removed.
Bearing this in mind, the proof follows as above up to conditions \eqref{eq:ineqL} and \eqref{eq:ineql}, 
taking however subsequences of $(t_n)$ and $(s_n)$ (if needed) such that
$$ \lim \int_{t_n-\sigma(t_n)}^{t_n}  \beta(s)\, \dd s = \zeta_1,\q
\lim \int_{t_n-\sigma(t_n)}^{t_n} \ \beta(s)\, \dd s = \zeta_2$$
for some $\zeta_1, \zeta_2 \in [\zeta_m, \zeta_M]$, with $\zeta_M$ as in \eqref{cond:las} and $\zeta_m := \liminf_{t \to \infty} \int_{t-\sigma(t)}^{t}  \beta(s)\, \dd s$. In this way, for a solution $x(t)$ with $l<L$ in \eqref{lL}, the inequalities \eqref{eq:ineqL} and \eqref{eq:ineql} become
\begin{subequations} \label{eq:ineqlL_nolimit}
    \begin{equation}
        L\le e^{-\delta\zeta_1}K+\frac{p}{\delta}L e^{-a  \tilde l}(1-e^{-\delta\zeta_1})
    \end{equation}
    \begin{equation}
        l\ge e^{-\delta\zeta_2}K+\frac{p}{\delta}l e^{-a \tilde L}(1-e^{-\delta\zeta_2}).
    \end{equation}
\end{subequations}

As  in Step 2, one sees that it is not possible to have $\zeta_1=\zeta_2=0$. If $\zeta_1>0$ and $\zeta_2=0$, proceeding as in Step 2   would lead to $l=l_\sigma=L_\sigma=K$ and $\tilde l\le K$, thus also $L=K$, which contradicts the assumption  $l<L$. Thus, we  deduce that $\zeta_i>0\, (i=1,2)$ and, arguing as above, that
$l\le \tilde l<K<\tilde L\le L$.

We now define
\begin{equation*}
    h_i(x)=\frac{e^{-\delta\zeta_i}K}{1-\mu_i f(x)}, \qq i=1,2,
\end{equation*}
for $x>f^{-1}(1/\mu_i)$,
where $\mu_i=1-e^{-\delta\zeta_i}$. 
Let
\begin{equation*}
    j(x, \zeta) := \frac{1-\mu}{1-\mu f(x)} \q{\rm where}\q \mu=\mu(\zeta)=1-e^{-\de \zeta},
\end{equation*}
and note that 
$
(x-K)   \frac{\p j}{\p\zeta} (x, \zeta)  <0$  for all $ x\ne K$.
Thus $\zeta \mapsto j(x, \zeta)$ is increasing for $x<K$ and decreasing for $x>K$, hence, for all $x$ such that $f(x)\mu(\zeta_M)<1$, it follows   that
\begin{subequations}\label{h1,h2}
    \begin{equation}
        h_1(x)=K j(x, \zeta_1) \le K j(x, \zeta_M) \, ,\q {\rm if}\q  x<K,
    \end{equation}
    \begin{equation}
        h_2(x)=K j(x, \zeta_2) \ge K j(x, \zeta_M)\, ,\q {\rm if}\q  x>K.
    \end{equation}
\end{subequations}
Define
\begin{equation}\label{h(x)}
    h(x)=\frac{e^{-\de \zeta_M}K}{1-(1-e^{-\de \zeta_M})f(x)}=K j(x, \zeta_M).
\end{equation}
From \eqref{eq:ineqlL_nolimit},
\begin{subequations}
    \begin{equation*}
        L\le h_1( \tilde l)\le h(l),
    \end{equation*}
    \begin{equation*}
        l\ge h_2(\tilde L) \ge h(L).
    \end{equation*}
\end{subequations}
From this point onwards, we can resume the proof of Step 3, and conclude that there are no solutions $x(t)$ with $l<L$ in \eqref{lL}.
The proof that $K$ is a global attractor is complete.
\end{proof}


\begin{rmk}\label{rmk4.2} {\rm For \eqref{Nich_MR} with  $p>\de>0$ and constant delays $\tau\ge\sigma\ge 0$, it was shown in \cite{MR18} that $K$ is a global attractor of all positive solutions if $(e^{\de \sigma}-1)\log(p/\de)\le 1$. This criterion  is a very particular case of the result established in Theorem \ref{thmGA_K1}.}
\end{rmk}

Combining  Corollary \ref{corLAS} and Theorem \ref{thmGA_K1}, we obtain:

\begin{thm}\label{thmGAS_K1}  
With the above notations,  let $p>\de$ and assume that one of the following conditions holds:

(i)  $\log(p/\de)\le \frac{2\de \zeta_M}{e^{\de \zeta_M}-\de \zeta_M-1}$ and $\de \zeta_M(2+\log(p/\de))<1$;

(ii)  $\log(p/\de)> \frac{2\de \zeta_M}{e^{\de \zeta_M}-\de \zeta_M-1}$ and $(e^{ \delta\zeta_M} -1)\log \frac{p}{\delta}  \le 1.$\\
Then the equilibrium $K$   of \eqref{eq:nicholson_1pair} is globally asymptotically stable (in $C_0^+$).
\end{thm}

\begin{proof} Define $c=\frac{2\de \zeta_M}{e^{\de \zeta_M}-\de \zeta_M-1}$, $g_1(x)=(e^x-1)\log(p/\de), g_2(x)=(2+\log(p/\de))x$. Since $g_1(\de \zeta_M)\le g_2(\de \zeta_M)$ if and only if  $\log(p/\de)\le c$, the latter condition and \eqref{cond:las} imply  (h1), and thus $K$ is stable and globally attractive;  if $\log(p/\de)>c$ and  (h1) are satisfied, then \eqref{cond:las} holds as well.
\end{proof}

\begin{rmk}\label{rmk4.3} {\rm The  global attractivity  in Theorem  \ref{thmGA_K1} suggests that  the criterion for local stability in  Corollary \ref{corLAS} is not sharp.
In fact, depending on the size of the coefficients in \eqref{eq:nicholson_lin},   sharper results for the local asymptotic stability of $K$ are occasionally obtained from the criteria in  \cite{bb2019}.}
\end{rmk}

\subsection{The case of $m$ pairs of delays}

Next, we address the global attractivity of the positive equilibrium $K$ of \eqref{eq:nicholson}.

\begin{thm}\label{thmGA_K2} Consider \eqref{eq:nicholson} and define $\zeta_M$ as in \eqref{zeta_M}.
Assume $p>\de$ and that:
\begin{itemize}
\item[(H1)] $\dps\frac{a^+}{a^-}< \frac{3}{2} $;
\item[(H2)] $a^+K(e^{\delta\zeta_M} -1)\le 1$.
\end {itemize}
Then the equilibrium $K$   of \eqref{eq:nicholson} is globally attractive.
\end{thm}

\begin{proof} The main arguments are as in the proof of Theorem \ref{thmGA_K1}, thus some details will be omitted. 

Let $x(t)$ be a solution with initial condition in $C_0^+$. If there exists $C=\lim_{t\to \infty}x(t)$, proceeding as in the above mentioned theorem it follows that $C=K$.

Define $l,L$ as in \eqref{lL}, and suppose that $l<L$.
Since afterwards one may proceed as in Step 4 above and obtain \eqref{h1,h2}, \eqref{h(x)},  without loss of generality   assume already that there exist
$$\zeta_j:= \lim_{t \to \infty} \int_{t-\sigma_j(t)}^{t}  \beta(s)\, \dd s,\qq j=1,\dots,m.$$

{\it  Step 1}.  We argue along the lines of Step 1 of the above proof. Consider sequences $(t_n),(s_n)$ with $t_n,s_n\to \infty, x(t_n)\to L,x(s_n)\to l,  x'(t_n)=x'(s_n)=0$,  such that (by taking  subsequences, if necessary),
  for each $j\in \{1,\dots,m\}$, 
\begin{equation}\label{limitsj_h,g}
\begin{split}
    &x(t_n -\tau_j(t_n ))\to L_{\tau_j}, \  x(t_n -\sigma_j(t_n ))\to L_{\sigma_j} , \ \\
     &x(s_n -\tau_j(s_n ))\to l_{\tau_j}, \ x(s_n -\sigma_j(s_n ))\to l_{\sigma_j}.
    \end{split}
\end{equation}
  As before, from  $x'(t_n)=x'(s_n)=0$ and taking limits, we deduce that 
  \begin{equation*}
    L=\frac{1}{\delta}\sum_j p_j L_{\tau_j} e^{-a_j L_{\sigma_j}} \q {\rm and}\q   l=\frac{1}{\delta}\sum_j p_j l_{\tau_j} e^{-a_j l_{\sigma_j}},
\end{equation*}
and consequently
$$\min_j L_{\sigma_j}\le K\le \max_j l_{\sigma_j},$$
which leads to $l\le K\le L$. Moreover, for each $ j=1,\dots,m$, there are $\tilde l_{ji},\tilde L_{ji}\in [l,L]\ (1\le i\le m)$ for which
\begin{equation}\label{eq:ineqlL2}
\begin{split}
&L\le e^{-\delta\zeta_j}L_{\sigma_j}+\frac{1}{\delta}L \sum_{i=1}^mp_ie^{-a_i  \tilde l_{ji}}(1-e^{-\delta\zeta_j}),\\
&   l\ge e^{-\delta\zeta_j}l_{\sigma_j}+\frac{1}{\delta}  l \sum_{i=1}^mp_ie^{-a_i \tilde L_{ji}}(1-e^{-\delta\zeta_j}).
\end{split}
\end{equation}
For 
$$\min_{1\le i, j\le m}\tilde l_{ji}=:\tilde l,\q\max_{1\le i,j\le m}\tilde L_{ji}=:\tilde L,$$
from the inequalities in  \eqref{eq:ineqlL2} we deduce that
\begin{equation}\label{eq:ineqlL3}
\begin{split}
&L\le e^{-\delta\zeta_j}L_{\sigma_j}+L(1-e^{-\delta\zeta_j})f(\tilde l),\\
&   l\ge e^{-\delta\zeta_j}l_{\sigma_j}+  l (1-e^{-\delta\zeta_j}) f(\tilde L), \qq j=1,\dots ,m,
\end{split}
\end{equation}
where, as before,
$\dps f(x)=\frac{1}{\delta} \sum_{i=1}^mp_ie^{-a_i  x},\ x\ge 0$.

{\it  Step 2}.  We claim that $l<K<L$.

As for the case $m=1$, one easily shows that $\tilde l\le K\le \tilde L$. Next suppose that $\tilde l=K$, so that $f(\tilde l)=1$. From \eqref{eq:ineqlL3}, one  obtains  $L_{\sigma_j}=L$ for all $j$, thus $L=\tilde L=K$.   Inserting $L=K$ in the second inequality of \eqref{eq:ineqlL3} for $i$ such that $l_{\sigma_i}=\max_{1\le j\le m} l_{\sigma_j}$, since $l_{\sigma_i}\ge K$ we arrive to
$$l\ge e^{-\de \zeta_i} l_{\sigma_i}+l(1-e^{-\de \zeta_i})\ge e^{-\de \zeta_i}K+l(1-e^{-\de \zeta_i})\ge l,$$
thus $l=K$, which is not possible. A similar contradiction is obtained if $\tilde L=K$. Therefore,
$l\le \tilde l< K< \tilde L\le L$.

{\it  Step 3}. Define
$$h_j(x)=\frac{e^{-\de \zeta_j}K}{1-(1-e^{-\de \zeta_j})f(x)} \q {\rm for}\q x>\th_1,j=1,\dots,m,$$
and
$$h(x)=\frac{e^{-\de \zeta_M}K}{1-(1-e^{-\de \zeta_M})f(x)} \q {\rm for}\q x>\th_1,$$
where $\th_1>0$ is such that $f(\th_1)=(1-e^{-\de \zeta_M})^{-1}$. 
From \eqref{eq:ineqlL3}, for $j$ such that $ L_{\sigma_j}\le K$, we have
$$L\le h_j(\tilde l)\le h_j(l) \le h(l);$$
and for $j$ such that $ l_{\sigma_j}\ge K$, we obtain
$$ l\ge h_j(\tilde L)\ge h_j(L)\ge h(L). $$
From this point onward, to resume 
the proof of Theorem \ref{thmGA_K1} it suffices to show that our hypotheses imply the Claims (i)-(iii) below.
\smal

 Claim (i):  $Sf(x)<0$ for $x>0$. 

In fact,
by virtue of assumption (H1),
\begin{equation*}
\begin{split}
Sh(x)&=Sf(x)=\frac{\sum_jp_j a_j^3 e^{-a_jx}}{\sum_jp_j a_je^{-a_jx}}-
\frac{3}{2}\left (\frac{\sum_jp_j a_j^2 e^{-a_jx}}{\sum_jp_j a_je^{-a_jx}}\right)^2\\
&= \frac{1}{2(\sum_jp_j a_je^{-a_jx})^2}\left (
2(\sum_jp_j a_j^3 e^{-a_jx})(\sum_jp_j a_je^{-a_jx})-3
(\sum_jp_j a_j^2 e^{-a_jx})^2\right )\\
&= \frac{\sum_jp_j e^{-a_jx}}{2(\sum_jp_j a_je^{-a_jx})^2}\sum_{j,i}p_jp_ie^{-(a_j+a_i)x}a_j^2a_i(2a_j-3a_i)<0.
\end{split}
\end{equation*}

\smal

Claim (ii): $|h'(K)|\le 1$.

The assertion follows by observing that
$h'(K)=K (e^{\de \zeta_M}-1)f'(K),$
thus  (H2) leads to
\begin{equation*}
\begin{split}
|h'(K)|&=  K (e^{\de \zeta_M}-1)\frac{1}{\de}\sum_jp_ja_je^{-a_j K}\\
&\le  (e^{\de \zeta_M}-1)a^+K\le 1.
\end{split}
\end{equation*}



Claim (iii):
 $\th:=Ke^{-\de \zeta_M}>\th_1$.
 
This claim is equivalent 
to $f(\th)< (1-e^{-\de \zeta_M})^{-1}$.
From (H2), $\th=Ke^{-\de \zeta_M}\ge K-\frac{1}{a^+}e^{-\de \zeta_M}$ and, since
 $f$ is decreasing,
 \begin{equation}\label{theta}
\begin{split}
(1-e^{-\de \zeta_M})f(\th)&\le (1-e^{-\de \zeta_M})f\big(K-\frac{1}{a^+}e^{-\de \zeta_M}\big)\\
&=(1-e^{-\de \zeta_M})\frac{1}{\de}\sum_jp_j\exp\left (-a_jK+\frac{a_j}{a^+}e^{-\de \zeta_M}\right)\\
&\le(1-e^{-\de \zeta_M})\exp\big (e^{-\de \zeta_M}\big).
\end{split}
\end{equation}
%
%
%
Next, again observe that
$$ \log(1-e^{-\de \zeta_M})+e^{-\de \zeta_M}<0,
$$
since $x+\log (1-x)<0$ for all $x\in (0,1).$ This proves  Claim (iii), ending the proof.
\end{proof}

\begin{rmk} {\rm Clearly  the case $m=1$ in Theorem \ref{thmGA_K1} is recovered by Theorem \ref{thmGA_K2}. To this end, as done in Theorem \ref{thmGAS_K1},  the assumptions in Theorem \ref{thmGA_K2}  together with \eqref{eq:las_K} or \eqref{eq:las_K2} provide sufficient conditions for the global asymptotic stability of $K$.}\end{rmk}
 \begin{rmk}\label{rmk4.6} {\rm Clearly, Theorem \ref{thmGA_K2} could be stated under slightly weaker hypotheses, as long as they allow to conclude that Claims (i)--(iii) in the proof above are satisfied.}\end{rmk}

\begin{rmk} {\rm The equation  $x'(t)=- \delta x(t)+ p x^\ga(t-\tau)e^{-a x(t-\tau)}$ with $\ga\in (0,1)$ has been applied as an economic model, where $\tau$ takes into account the time delay in the process of production of capital or in the reaction to market changes and fluctuations. Recently, the model
\begin{equation}\label{eq:nicholson_ga}
    x'(t)=\beta (t) \bigg( \sum_{j=1}^{m} p_j  x^\ga(t-\tau_j(t)) e^ {-a_j x(t-\sigma_j(t))}-\delta x(t)\bigg)
\end{equation}
with $\ga\in (0,1)$ was  studied in \cite{huang20}. It is worth emphasising  that there is always a positive equilibrium $K$  for \eqref{eq:nicholson_ga} given by the equation $\sum_jp_je^{-a_jK}-\de K^{1-\ga}=0$. Assuming several constraints,  among them that $a^+K(e^{\de \be^+\tau}-1)\le 1$ -- which is a stronger version  of our hypothesis (H2) --,  Huang et al. \cite[Theorem 3.3]{huang20} showed that $K$ is a global attractor for \eqref{eq:nicholson_ga}. We stress that the case $\ga=1$ in \eqref{eq:nicholson_ga} was not covered in \cite{huang20}.}\end{rmk}

 
 Since in general $K$  is not explicitly computable when there are multiple pairs of delays,  a criterion with (H2)  not depending on $K$ is useful. 

\begin{cor}\label{corGA_K2} With the previous notations,
assume $p>\de$, (H1) and
\begin{itemize}
\item[(H2*)] $\dps  \frac{a^+}{a^-}(e^{\delta\zeta_M} -1)\log \frac{p}{\delta} \le 1$.
\end {itemize}
 Then the equilibrium $K$   of \eqref{eq:nicholson} is globally attractive.
\end{cor}

\begin{proof}
From $f(K)=1$, we get  $a^-K\le \log(p/\de)\le a^+K
$ and consequently (H2*) implies (H2).
\end{proof}

\begin{exmp}\label{ex1} {\rm Inspired by an example in \cite{huang20}, consider \eqref{eq:nicholson} with $m=2$,
$p_1=\frac{3}{50}e^{4},p_2=\frac{1}{25}e^{5}, a_1=\frac{4}{5},a_2=1, \de =0.1$,   and $\be(t),\tau_j(t),\sigma_j(t)\, (j=1,2)$ continuous, non-negative and bounded, with $\be(t)\ge \be^->0$ on $[0,\infty)$. In this situation, the positive equilibrium is explicitly computed as $K=5$.
With the above notations,  $f(x)=\frac{3}{5}e^{4(1-x/5)}+\frac{2}{5}e^{5-x}$, $\th=5e^{-0.1 \zeta_M}$,
$\frac{a^+}{a^-}=\frac{5}{4}<\frac{3}{2}$ and  $|h'(5)|=4.4(e^{0.1 \zeta_M}-1)$. Hence $|h'(5)|\le 1$ if
\begin{equation}\label{zetaExpl1}\zeta_M\le  10\log (\frac{27}{22} )\approx 2.05.\end{equation}
If \eqref{zetaExpl1} holds, then
$(1-e^{-0.1 \zeta_M})f(\th)\le (1-\frac{22}{27})f(\th)\le \frac{1}{27}(3e^{\frac{20}{27}}+2e^{\frac{25}{27}})\approx 0.42<1$. Theorem \ref{thmGA_K2} and Remark \ref{rmk4.6} allow  concluding that $K=5$ is a global attractor of all positive solutions. On the other hand, note that (H2) reads as $\zeta_M\le  10\log (1.2 )\approx 1.82$, a more restrictive imposition.

For instance, choose e.g. $\be(t)=1+\sin^2 t, \sigma_j(t)=|\cos jt|$ and any $\tau_j(t)$ satisfying the general conditions  above, $j=1,2$. Clearly $\zeta_M\le 2$, thus \eqref{zetaExpl1} holds. With the same $\be(t)$ and $\sigma_j(t)=2j\pi\, (j=1,2)$, one  sees that
$\zeta_M=3/2$, thus $\de \zeta_M(2+\log(p/\de))\approx0.98 <1$. From Corollary \ref{cor3.1}, $K$ is globally asymptotically stable.
 }
\end{exmp}
%


\begin{exmp}\label{ex1} {\rm Consider again \eqref{eq:nicholson} with $m=2$, but now suppose that
$p_1=\frac{3}{50},p_2=\frac{1}{25}, a_1=\frac{4}{5},a_2=1, \de =0.09$,   and $\be(t),\tau_j(t),\sigma_j(t)\, (j=1,2)$ satisfy the general conditions  above. In this situation, $p=0.1>0.09$, but the positive equilibrium $K$ is not explicitly given. Instead, we can assert that (H2*) is satisfied if
$$
\zeta_M\le  \frac{100}{9}\log \bigg(1+\frac{4}{5\log \frac{10}{9}} \bigg)\approx 23.90,$$
in which case $K$ is globally attractive.
}
\end{exmp}

\section{Conclusions}
\setcounter{equation}{0}

In the present paper, we have considered a  Nicholson-type equation with multiple pairs of delays, Eq.~\eqref{eq:nicholson}. The presence of multiple mixed monotone nonlinear terms, each one with two different  time-varying delays, may drastically alter the  dynamics and stability of  equilibria of either the associated differential equation without delays, or the DDE with only one delay $\tau_j(t)=\sigma_j(t)\ (1\le j\le m)$ in each nonlinear term \cite{bb2017}.

In \cite{NE_mult}, the authors showed  the global stability of the equilibrium 0 when 
$p:=\sum_j p_j\le \de$.
Here, we have assumed $p>\de$, proved the permanence without additional conditions, and studied the  local stability and global attractivity of the positive equilibrium $K$ of  \eqref{eq:nicholson}. Inspired by the techniques in \cite{MR18}, we provided sufficient conditions for the global attractivity of $K$ which are easily verifiable, see the main result Theorem \ref{thmGA_K2}.  Our criteria do not involve the delays $\tau_j(t)$, depend on the size of the delays $\sigma_j(t)$ in the sense that the limit $\zeta_M$ in \eqref{zeta_M} must not be too large so that (H2) holds (or Claims (ii)-(iii) in the proof of Theorem \ref{thmGA_K2}), and impose a constraint expressed by (H1) on the relative sizes of the several  coefficients $a_j$. Whether one can replace (H1), or alternatively of requirement of $Sf(x)$ negative, by a weaker condition is a question deserving further analysis.
The results presented here  largely generalise the ones obtained in \cite{MR18} for \eqref{Nich_MR}.
 Also, they provide an  answer to an open problem  in \cite{bb2017}, not only concerning the model  \eqref{eq:nicholson_1pair} with  one pair of delays, but the more general equation \eqref{eq:nicholson} with $m$ pairs of time-varying delays. 
 
 The achievements in this paper  also raise some interesting problems.  
 We believe that the application of the present method  would lead to original results if  adopted to other Nicholson-type equations, such as  the model \eqref{eq:nicholson_ga} with $\ga\in (0,1)$, studied in  \cite{huang20}.
 If instead of \eqref{eq:nicholson} one considers the  general nonautonomous model
 \begin{equation}\label{general}
    x'(t)=\sum_{j=1}^{m} p_j(t)  x(t-\tau_j(t)) e^ {-a_j(t) x(t-\sigma_j(t))}-\delta(t) x(t),
\end{equation}
 it is natural to inquire about sufficient conditions for extinction, permanence or global  stability of positive solutions of  such equations. These will be the aims of our future work, where,
since the procedure developed here  is not  applicable  to \eqref{general}, we intend to  exploit different techniques such as the ones proposed in   \cite{Faria21}. The treatment of mixed monotonicity Nicholson systems with patch structure, given by $n$-dimensional versions of \eqref{eq:nicholson}, is clearly a challenge deserving attention and  another strong motivation for  future research.

\section*{Acknowledgements}
This work was  supported by the National Funding from FCT - Funda\c c\~ao para a Ci\^encia e a Tecnologia (Portugal) under project UIDB/04561/2020.

%
%
%
%

\end{document}